\documentclass[11pt,a4paper,oneside]{article}
\usepackage{amscd,amsfonts,amsmath,amssymb,amstext,amsthm,graphicx,xcolor,mathrsfs,a4wide,tikz,framed,oubraces,
wrapfig}
\usetikzlibrary{matrix,arrows,shapes,snakes,patterns}
\usepackage{multirow,arydshln}
\usepackage{makeidx}
\makeindex %This makes the index
\usepackage[totoc]{idxlayout}
\usepackage{algorithm}
\usepackage[noend]{algpseudocode}

\def\Sym{\mathop{\mathrm{Sym}}\nolimits}
\def\Mat{\mathop{\mathrm{Mat}}\nolimits}

\def\CC{\mathbb{C}}

\def\frakg{\mathfrak{g}}

\def\frakk{\mathfrak{k}}

\def\frakq{\mathfrak{q}}
\def\fraks{\mathfrak{s}}

\def\fraku{\mathfrak{u}}

\makeatletter
\def\BState{\State\hskip-\ALG@thistlm}
\makeatother
\title{Pseudoconcavity of flag domains:\\
The method of supporting cycles}
\author{T. Hayama, A. Huckleberry and Q. Latif}
\date{}
\theoremstyle{plain}
\newtheorem{theorem} {Theorem}
[section]
\newtheorem{lemma} [theorem]{Lemma}
\newtheorem{proposition}[theorem]{Proposition}
\newtheorem{corollary} [theorem]{Corollary}
\theoremstyle{definition}
\newtheorem*{definition} {Definition}
\newtheorem*{example} {Example}
\begin{document}

\title{Pseudoconcavity of flag domains:\\ The method of supporting cycles
\thanks{ \noindent
The first author was supported by JSPS KAKENHI Grant Number 16K17576.
The third author was supported in the project by the a grant in the DFG-Schwerpunkt 1388, "Darstellungstheorie". 
Part of his Jacobs University thesis  (\cite{L}), which was guided by the second author and strongly influenced by the first author, was devoted to
the results in the present paper.  The first and second author's joint work was supported by funds from the 
Korean Institute of Advanced Study (KIAS) and  the NSF during their visits to Seoul, respectively Berkeley.  We are all thankful for this support.}
}

\date{Received: date / Accepted: date}

\maketitle

\begin{abstract}
\noindent
A flag domain of a real from $G_0$ of a complex semismiple Lie group $G$ is an open
$G_0$-orbit $D$ in a (compact) $G$-flag manifold.  In the usual way one reduces to
the case where $G_0$ is simple. It is known that if $D$ possesses non-constant holomorphic functions,
then it is the product of a compact flag manifold and a Hermitian symmetric bounded domain.
This pseudoconvex case is rare in the geography of flag domains.  Here it
is shown that otherwise, i.e., when $\mathcal{O}(D)\cong \mathbb{C}$, the flag domain
$D$ is pseudoconcave.  In a rather general setting the degree of the pseudoconcavity is
estimated in terms of root invariants.  This estimate is explicitly computed for domains
in certain Grassmannians.
%\keywords{flag domains \and pseudoconcavity \and cycles}
% 
%\subclass{14M15 \and 32M05  \and 57S20}
\end{abstract}
\noindent
\section{Introduction}
Throughout $G_0$ denotes a (connected) semi-simple Lie group which is
acting on a flag manifold $Z=G/Q$ of its complexification $G$.  If $G_0$ is
not simple, then all objects considered here split into products according to
a product decomposition of $G_0$ and therefore we assume that $G_0$ is
simple.  Since $G_0$ has only finitely many orbits on $Z$ (\cite{W}), it has open 
orbits $D$ which are referred to as \emph{flag domains}.   If $G_0$ is a compact
real form of $G$, then $D=Z$, a situation that is not of interest from the point of
view of this paper.  There are also exceptional situations where $G_0$ is a non-compact
real form, but nevertheless $D=Z$ (see, e.g., \cite{FHW} for this and other background).
We also assume there that this is also not the case, i.e., we only consider flag domains  $D$ 
which are proper open subsets of $Z$.  Here we focus on certain complex geometric properties
of these domains which involve basic compact complex submanifolds which are associated
to the group actions at hand.  We begin with background information followed by a statement
of the first version of our results (Theorem \ref{weak pseudoconcavity}).  

\subsection {Convexity and concavity}
\subsubsection {Generalities}
A complex space $X$ is a Stein space if and only if the holomorphic functions on $X$ separate its
points, and given any divergent sequence $\{x_n\}$ in $X$ there exists a holomorphic function
$f\in \mathcal{O}(X)$ with $\lim \vert f(x_n)\vert =\infty $.  If only the second condition is fulfilled, then
$X$ is said to be holomorphically convex.  In this case there is a canonical proper, surjective holomorphic
map $X\to X_{red}$ onto a Stein space.  At the level of sets this \emph{Remmert reduction} is defined
by the equivalence relation $x\sim y$ whenever $f(x)=f(y)$ for all $f\in \mathcal{O}(X)$. If $\mathcal{S}$
is a coherent sheaf on $X$, or more simply a holomorphic vector bundle, and $X$ is Stein, then
all cohomology groups $H^q(X,\mathcal{S})$ vanish for $q\ge 1$.  If $X$ is holomorphically convex,
then in principle these groups can be computed from the Remmert reduction; in particular, they are
finite-dimensional.  Since many problems in complex analysis can be stated in terms of cohomological
invariants, these results are of fundamental importance (see, e.g., \cite{GR1}, \cite{GR2}).

\bigskip\noindent
The role of convexity in Stein theory can be exemplified as follows: Suppose $\Omega $ is a relatively 
compact domain with smooth boundary in a complex manifold $X$, e.g., $X=\mathbb{C}^n$.   Let us suppose 
that every point $p\in \partial \Omega $ possesses a coordinate ball $B_p$ in $X$ in which there
is a $1$-codimensional complex submanifold $H_p$ which contains $p$ but which otherwise is contained
in the complement of the closure of $\Omega $.  In this setting there exist holomorphic functions $h\in \mathcal{O}(B_p)$
which vanish exactly on $H_p$.  Thus, if $\{x_n\}$ is a sequence in $B_p\cap \Omega $ which converges to $p$ and $h$ is such a function,
then $f:=\frac{1}{h}$ is such that $lim \vert f(x_n)\vert =\infty $.  In this situation one says $\Omega $ has a supporting
holomorphic supporting hypersurfaces at each of its boundary points, a condition that is reminiscent of classical convexity.
Under further natural conditions on $X$, e.g., if $X$ itself is Stein, one can in fact show that if $\Omega $ is convex in this 
strong sense, then it is Stein.  

\bigskip\noindent
It is convenient to discuss supporting manifolds such as the hypersurfaces $H_p$ in terms of local holomorphic maps
of $d$-dimensional polydisks $\Delta =\Delta_d(r):=\{z\in \mathbb{C}^d: \vert z_i\vert \le r \ \text{for all}\  i\}$.  A $d$-dimensional
polydisk through $p$ in a complex manifold $X$ is a map $F:\Delta \to X$ which is holomorphic with $rank_z(dF)=d$
at all points $z\in \Delta $ and with $F(0)=p$.  On the convexity side we can localize the hypersurfaces $H_p$ discussed
above so that they are parameterized by such maps.   Andreotti defined the notion of \emph{pseudoconcavity} analogously.
The following definition is a bit stronger than Andreotti's (\cite{A}), but is appropriate for our present considerations.
\begin {definition}  
A connected complex manifold $X$ is said to be \emph{pseudoconcave of degree $d$} if it contains a 
relatively compact open set $Y$ so that for every $p\in \partial Y$ there exists a holomorphic map $F:\Delta_d(r)\to X$
through $p$ with boundary $F(\partial \Delta)$ contained in $Y$.
\end {definition}
\noindent
Andreotti's first theorem in this setting is that if $X$ is \emph{pseudoconcave}, i.e., of degree $d=1$, then
$H^0(X,\mathcal{S})$ is finite-dimensional for all coherent sheaves $\mathcal{S}$, i.e. finiteness of cohomology 
at level $0$. If $X$ can be smoothly exhausted by $\rho :X\to \mathbb{R} ^{\ge 0}$ whose sublevel sets 
$Y_r:=\{\rho <r\}$ are, for $r$ sufficiently large,  $d$-pseudoconcave, then finiteness of cohomology holds
for $H^q(X,\mathcal{S})$ for $q\le d-1$ (\cite{AG}). 

\subsubsection {Convexity and concavity of flag domains.}
Every Hermitian symmetric space of non-compact type (HSS) is an example of a flag domain $D_{SS}$.  Such
is canonically realized as an orbit of a real form $G_0$ (of Hermitian type) in a flag manifold $Z_{SS}$ which is the 
associated Hermitian symmetric  space of compact type.   Using this embedding of $D_{SS}$ one
realizes it as a bounded Stein domain in the orbit of an associated unipotent
group which is a copy of $\mathbb{C}^n$.   More generally, if $D$ is any $G_0$-flag domain in a
$G$-flag manifold $Z$ and $\mathcal{O}(D)\not=\mathbb{C}$, then $G_0$ is of Hermitian
type, $D$ is holomorphically convex and the Remmert reduction is given as the
(saturated) restriction of a $G$-equivariant projection $Z\to Z_{SS}$ to a $G_0$-fibration
$D\to D_{SS}$.   This situation is discussed in more detail below, but here let us
just note that the fiber $C$ of this bundle is a compact submanifold of $D$ which is canonically
associated to the group actions at hand. Since $D_{SS}$ is Stein and contractible,
the bundle is holomorphically trivial, i.e., $D$ is bundle equivalent to $D_{SS}\times C$.
Needless to say, the holomorphically convex flag domains are of a very special type.

\bigskip\noindent
The following is a first (qualitative) version of the main result of this work.
\begin {theorem}\label{weak pseudoconcavity}
A flag domain is pseudoconcave if and only if it is not holomorphically convex.
\end {theorem}
\noindent
The sets $Y$, which display the pseudoconcavity of $D$, are constructed in such a way
that in many situations  the dimensions of the  supporting polydisks can be explicitly computed
in terms of certain root theoretic invariants (see Corollary \ref{pseudoconcave}).  This leads to refined estimates
on the degree $d$ of the pseudoconcavity.  Since the sets $Y$ can be realized as sublevel sets
of (not necessarily smooth) exhaustions,  one would hope that finiteness theorems for higher
degree cohomology can be proved.  

\bigskip\noindent
In ([Hu]) it was shown that a flag domain $D$ possesses non-constant holomorphic functions if and only if it is pseudoconvex.  This condition
is also equivalent to $D$ being holomorphic convex.  It was conjectured there that otherwise $D$ should be pseudoconcave.  With this in mind
we have the following reformulation of Theorem \ref{weak pseudoconcavity}.
\
\begin {corollary}
A necessary and sufficient condition for a flag domain to be pseudoconcave is that it is not pseudoconvex.
\end {corollary}

\section {Background on cycles and cycle spaces}
The supporting polydisks for the realization of the pseudoconcavity of a given flag domain $D$
are contained in important compact submanifolds of $D$ which are defined by the group actions
at hand.  The starting point for the discussion is the fact that every maximal compact subgroup $K_0$ of $G_0$
has a unique orbit $C_0$ in $D$ which is a complex submanifold.  Such  \emph{base cycles} can also
be characterized as the unique compact orbits in $D$ of the complexification $K$ of $K_0$.   A special
case of Matsuki duality (\cite{W}) states that there is a $1-1$ correspondence between the  closed
$K$-orbits $C_0$ in $Z$ and the flag domains which contain them.  

\bigskip\noindent
Letting $q:=dim_{\mathbb{C}} C_0$, we view $C_0$ as being moved by the algebraic $G$-action on the
space $\mathcal{C}_q(Z)$ of $q$-dimensional cycles in $Z$.  For our purposes here it is enough to
consider the orbit $\mathcal{M}_Z:=G.[C_0]$ and its open subset $\mathcal{M}_D$ which is defined to
be the connected component containing $[C_0]$ of the set of cycles in $\mathcal{M}_Z$ which are contained
in $D$.  Clearly $K$ is contained in the isotropy group $G_{[C_0]}$, but this isotropy group can be larger:
If $Q=G_p$ is the isotropy group at a point $p\in C_0$, it is possible that $Q$ also stabilizes $C_0$ and
that the parabolic group $P:=QK$ is $G_{[C_0]}$ (see \cite{W} and \cite{FHW} for this and other related background
material).\\

\noindent
\textbf{Remark}\label{isotropy}
Below we describe the two ways that this can occur.  Both are in the case where $G_0$ is of Hermitian type,
i.e., where the center of $K_0$ is $1$-dimensional.  As background,  let us comment that In both the Hermitian and non-Hermitian cases $K_0$
is a maximal subgroup of $G_0$ (see \cite{Hu2} for a proof).  In the non-Hermitian case the complex Cartan decomposition $\mathfrak{g}=\mathfrak{k}\oplus \mathfrak{s}$
is such that the $\mathfrak{k}$-representation on $\mathfrak{s}$ is irreducible. Therefore the connected complex group $K$ is maximal in the
sense that no proper Lie subgroup of $G$ contains it properly.  Since $K$ is algebraic, its normalizer is likewise algebraic and consequently in the non-Hermitian case
the only proper subgroups which contain $K$ are finite extensions of it.   In the Hermitian case $\mathfrak{s}=\mathfrak{s}^-\oplus \mathfrak{s}^+$
splits into two irreducible $\mathfrak{k}$-modules with $\mathfrak{p}^-=\mathfrak{k}\oplus \mathfrak{s}^- $ and $\mathfrak{p}^+=\mathfrak{k}\oplus \mathfrak{s}^+$
being parabolic. Other than $G$, the associated parabolic subgroups $P^-$ and $P^+$ are the only Lie subgroups which properly contain $K$. In this case, finite extensions
lying between $K$ and such a parabolic $P$ are trivial, because $P/K$ is isomorphic to $\mathbb{C}^n$ which does not support a fixed point free action of
a finite group (see \cite{W} and \cite{FHW} for this and other related background material).\qed

\bigskip\noindent
\begin{example}
Recall that if $G_0$ is of Hermitian type and $D_{SS}$ is one of its Hermitian 
symmetric spaces of non-compact type,  then $D_{SS}$ is a $G_0$-flag domain in its compact 
dual $Z_{SS}$.  Since $D_{SS}$ is Stein, $C_0$ consists of just one point $p_0$ and therefore
$G_{[C_0]}=G_{p_0}=:P$ with $Z_{SS}=G/P$ and $K\subset P$. Now suppose that $Z$ is another flag manifold with
$G$-equivariant projection $\pi :Z\to Z_{SS}$ and let $D$ be a flag domain with $\pi(D)=D_{SS}$.  For the
same reason as above, the base cycle $C_0$ is mapped to a point $p_0$ with $G_{p_0}=:P$ containing
$K$. It also contains the isotropy subgroup $Q=G_{q_0}$ of any point $q_0\in C_0$, because for every $g\in Q$ Remmert's
proper mapping theorem implies that 
$\pi (g(C_0))$ is a connected compact subvariety containing $p_0$ in $D_{SS}$.  Since $D_{SS}$ is holomorphically
separable, it follows that $\pi (g(C_0))=\{p_0\}$. Thus it follows that
that $P$ is either $P^+$ or $P^-$ with $K.Q=P$. In other words
$D$ is the full $\pi $-preimage of $D_{SS}$ and the $C_0$ is just the fiber over the base point $p_0$.
In this case $\mathcal{M}_D=D_{SS}$ and $\mathcal{M}_Z=Z_{SS}$.  In this special situation we will
say that the flag domain $D$ \emph{fibers over a HSS}.
\end {example}

\bigskip\noindent
There is one other special situation where the isotropy $G_{[C_0]}$ is essentially larger than $K$.
 Here is a concrete example of how this occurs.
 
 \bigskip\noindent
 \begin {example}
 %\textbf{Example 2.} 
 Let $G_0:=SU(1,n)$ act on $Z=\mathbb{P}(\mathbb{C}^{n+1})$ as usual.   It has
 two open orbits, the set of negative $1$-dimensional subspaces, which we denote by $D_-$ 
 and which is biholomorphically equivalent to the unit Euclidean ball in $\mathbb{C}^n$, and
 the set $D_+$ of positive subspaces, i.e., the complement of the closure of $D_-$.  Let 
 $E_-:=Span\{e_1\}$, $E_+:=Span\{e_2,\ldots ,e_{n+1}\}$ and $K_0$ be the stabilizer in $G_0$ of the
 decomposition $\mathbb{C}^{n+1}=E_-\oplus E_+$.  The base cycle in $D_-$ is just the $K_0$-fixed
 point $\mathbb{P}(E_-)$ and the base cycle in $D_+$ is the projective hyperplane $C_0:=\mathbb{P}(E_+)$.
 The Borel subgroup $B$ of upper triangular matrices stabilizes $C_0$ so that $P:=BK$ is its full stabilizer.
 (This is of course just the isotropy group of the associated point in the dual projective space.) 
 \end {example}

 \bigskip\noindent
 The following is the general setting of the above example.

 \bigskip\noindent
 \textit{Example 2 (general).}  Assume that $G_0$ is of Hermitian type.  Suppose that $D$ is a flag domain in $Z$ with base cycle $C_0$ and let $p_0\in C_0$.
 Let $\pi :\hat{Z}\to Z$ be the projection from the manifold of full flags and $\hat{D}$ be a flag domain in
 $\hat{Z}$ which holomorphically maps to $D_{SS}$ (see Example 1.).  Assume that $\pi (\hat{D})=D$.  Now let $B$ be the Borel subgroup of $G$ defined by a point 
 $q_0$ in the base cycle $\hat{C}_0$ with $\pi (q_0)=p_0$.  Since $\hat{D}$ maps to the HSS $D_{SS}$ and, as we have seen above,
 $P=BK$ is the isotropy subgroup of the $1$-point cycle in $D_{SS}$, it follows in particular that $B$ stabilizes $\hat{C}_0$ and that it therefore  stabilizes $C_0$.
 In fact  $P:=B.K$ is its full stabilizer.  In this case $\mathcal{M}_Z$ is again $Z_{SS}=G/P$ and $\mathcal{M}_D$ is the HSS flag domain $D_{SS}$.
 %\qed
%
 
 \bigskip\noindent
 In Example 2 (general) the flag domain $D$ does not necessarily fiber over a HSS, but the phenomenon is more or less the same: $D$ lifts to the
 canonical domain $\hat{D}$ in the manifold of full flags which then maps down to the Hermitian symmetric space $D_{SS}$.  For the purposes of our work
 here it is only important that $\mathcal{M}_D$ is the Hermitian symmetric space  $D_{SS}$ and that $\mathcal{M}_Z=Z_{SS}$.   We then refer to such $D$ as those having 
 $\mathcal{M}_D=D_{SS}$.

 \bigskip\noindent
 If $\mathcal{M}_D$ is not a HSS, e.g., if $G_0$ is not of Hermitian type, then, as indicated above, $G_{[C_0]}$ is at 
 most a finite extension of $K$ (see Remark \ref{isotropy}).  In that setting we abuse notation and write $G_{[C]_0}=K$ with  $\mathcal{M}_Z=G/K$ and $\mathcal{M}_D$ 
 being the corresponding open subset.   Thus we have the following dichotomy.
 \begin {proposition}
 Either $\mathcal{M}_D=D_{SS}$ as in the above examples or $\mathcal{M}_Z=G/K$.
 \end {proposition}    

\noindent
Let us conclude this paragraph by formulating the above in terms of tangent spaces of cycle spaces.
Recall that if $\mathcal{M}_D=D_{SS}=G_0/K_0$, then 
$T_{[C_0]}\mathcal{M}_Z$ is either isomorphic to $\mathfrak{s}^+$ or $\mathfrak{s}^-$ as a $K$-representation.  In particular, it is irreducible.
In the other case, i.e., where $\mathcal{M}_Z=G/K$, there are two possibilities:  If $G_0$ is not of Hermitian type, then the $K$-representation
on $T_{[C_0]}G/K$ is irreducible since $K$ is maximal subgroup of $G$.  However, if $G_0$ is of Hermitian type, then the cycle space $\mathcal{M}_D$ is the product of a $D_{SS}$ and its complex conjugate $\overline{D}_{SS}$ (\cite{BHH}).   In that case the representation of $K_0$ (resp. $K$) on the corresponding tangent space splits: 
$T_{[C_0]}\cong \mathfrak{s}^-\oplus \mathfrak{s}^+$.

\section {Related results}
In (\cite{Hu}) the notion of cycle connectivity was considered:  A flag domain $D$ is said to be \emph{cycle connected} if given any two points
$p,q\in D$ there exists a connected chain $\mathcal{C}=C_0+\ldots +C_N$ with $C_i\in \mathcal{M}_D$ so that $p\in C_0$ and $q\in C_N$. In particular it
was shown that $D$ is not holomorphically convex if and only if it is cycle connected.
One says that $D$ is  $k$-cycle connected if independent of $p,q$ the number $N$ is at most $k$.  In the $1$-connected case relatively compact neighborhoods
$Y$ of the base cycle $C_0$ can be constructed in such a way that every point of the closure of $Y$ is contained in a cycle $C_1$ which is connected to $C_0$
(\cite{Hu}).  In this way one shows that $D$ is pseudoconcave of degree one.

\bigskip\noindent
In (\cite{K}), in an even more general cycle connected setting, a finiteness theorem at the level of $H^0$ was proved.  It is interesting that an important component
of the argument in (\cite{K}) is reminiscent of Andreotti's application of Siegel's Schwarz-Lemma method for proving finiteness theorems for pseudoconcave
complex spaces.

\bigskip\noindent
In further related work  (\cite{HHL}) we study flag domains from the point of view of cycle connectivity.  In particular, applications of the method of
connecting points of $D$ to the base cycle by closures of orbits of distinguished unipotent $1$-parameter groups are given.  Certain of these results,
which were motivated by applications to Hodge theory, appear in the unpublished preprint (\cite{Ha}) and others in the thesis (\cite{L}).  Finally, in
(\cite{HHS}) the Levi geometry of $D$ is studied from the point of view of curvature invariants of the normal bundle of $C_0$.  In particular $1$-pseudoconcavity
is proved by computing the Levi-form of boundary of a neighborhoods of $C_0$ which are defined as sublevel sets of functions which are transported norm
functions on the normal bundle of $C_0$.
\section {Supporting cycles}
Here we construct the sets $Y$ which display the pseudoconcavity of a (non-holomorphically convex) flag domain $D$.
Every boundary point $p_0\in \partial Y$ is contained in at least one cycle $C$ which is contained in the closure $c\ell(Y)$ and which is
referred to as a supporting cycle at $p_0$.

\bigskip\noindent
In the setting of the previous section, let $\mathfrak{X}_Z:=\{(p,[C])\in Z\times \mathcal{M}_Z: p\in C\}$ and $\mu :\mathfrak{X}_Z\to Z$ and 
$\nu :\mathfrak{X}_Z\to \mathcal{M}_Z$ be the projections to the individual factors.
%
%\begin{center}
%\begin{tikzpicture}
 % \matrix (m) [matrix of math nodes,row sep=3em,column sep=4em,minimum width=2em] {
 %       \       &  \mathfrak{X}_Z    &   \    \\
 %       Z       &  \   &   \mathcal{M}_Z  \\};
  %\path[-stealth]
  %  (m-1-2) edge node [left] {$\mu$} (m-2-1)
  %          edge node [right] {\ $\nu$} (m-2-3);
% \end{tikzpicture}
%\end{center}
The maps $\mu $ and $\nu $ define homogeneous holomorphic $G$-bundles and are in particular open maps. Since the neutral fiber of $\nu $ is just the base cycle $C_0$, it is 
immediate that it is proper.  If we replace $Z$ by $D$ in this diagram, the resulting maps are still $G_0$-equivariant and open.  In this case it is known that $\mathcal{M}_D$ 
is a Stein manifold which is topologically a cell. Consequently, the bundle $\nu :\mathfrak{X}_D\to \mathcal{M}_D$ is trivial.

\subsection {Basic construction}
Here $B$ denotes any relatively compact domain with smooth boundary which contains the base point $[C_0]$ in $\mathcal{M}_D$.
Its preimage by the proper bundle map $\nu $ is denoted by $X$.  The relatively compact image in $D$ of $X$ by the open map $\mu $ is 
denoted by $Y$.  Its closure $c\ell(Y)$ is the $\mu $-image of $c\ell(X)$.  If $p_0\in \partial Y$, then there exists $[C]$ in $\partial B$ with
$(p_0, [C])\in \partial X$, i.e., every boundary point of $Y$ is contained in a supporting cycle. 
If $\mathcal{O}(D)\not\equiv \mathbb{C}$,
then $D$ fibers as a homogeneous bundle over a Hermitian symmetric space $\widehat{D}$ of non-compact type.  In fact the fibration $D\to \widehat{D}$
is another realization of $\nu:\mathfrak{X}_D\to \mathcal{M}_D$ and in that case $\mu $ is just the identity.  Thus in that case the supporting cycles
are completely contained in the boundary of $X$ and no pseudoconcavity observable in this construction.  On the other hand, we have the
following main result of this section.
\begin {theorem} 
\label{supporting cycle intersection}
If $D$ is a flag domain with $\mathcal {O}(D)\cong \mathbb{C}$, i.e., which does not fiber holomorphically nor anti-holomorphically over a Hermitian symmetric space
of bounded domain type, then for every supporting cycle $C$ the intersection $C\cap Y$ is open and dense in $C$.
\end {theorem}
\noindent
The proof is a consequence of the results in the following paragraph.  Note that the intersection $C\cap \partial Y$ of a supporting cycle and $\partial Y$
could contain positive dimensional analytic sets. In fact in general this is the case.  However, one can always find a holomorphic disk $\Delta $
in $C$ which is at least $1$-dimensional and which displays the pseudoconcavity of at the given point.

\bigskip\noindent
\textit{Proof of Theorem \ref{weak pseudoconcavity}.}  Let $p_0\in \partial Y$ and $C$ be a supporting cycle at $p_0$.  Recall that $C$ is a flag manifold of
some conjugate of $K$.  Consequently there exists a Borel subgroup $B$ of $K$ with $B. p_0$ Zariski dense in $C$.  From the above theorem it then
follows that intersection $B.p_0\cap Y$ is non-empty.   If we identify $B.p_0$ with $\mathbb{C}^n$ we may choose a complex affine line $A$ which contains
both $p_0$ and $p \in B.p_0\cap Y$ when regarded as an affine curve in $B.p_0$.  The normalization $\hat{A}$ of the closure of $A$ in $Z$ is a copy of $\mathbb{P}_1(\mathbb{C})$.
If $U$ is a sufficiently small disk containing $p_0$, then $c\ell(U)\subset Y$ and $\Delta:=\hat{A}\setminus U$ is a holomorphic disk which displays the
concavity at the boundary point $p_0$.\qed

\subsection {Cycle space isotropy representations}
We now continue toward the proof of Theorem \ref{supporting cycle intersection}. 
For this there are three cases that must be considered.  First, if $G_0$ is of Hermitian type, then $\mathcal{M}_D$ is either the associated symmetric space $D_{SS}$  or the  
product $D_{SS}\times \overline{D}_{SS}$ of $D_{SS}$ with its complex conjugate.  If $G_0$ is not of Hermitian type, then,
just as in the case where $G_0$ is Hermitian and $\mathcal{M}_D=D_{SS}\times \overline{D}_{SS}$, the cycle space 
$\mathcal{M}_D$ is a well-studied neighborhood of the real symmetric space $G_0.C_0\cong G_0/K_0$ in the affine symmetric space $\mathcal{M}_Z=G/K$.

\bigskip\noindent
The key ingredient for the main result in the following paragraph is that if $G_0$ is of Hermitian type and $\mathcal{M}_D\cong D_{SS}$ or $G_0$ is not of
Hermitian type, then the isotropy representation of $K$ on $T_{[C_0]}\mathcal{M}_D$ is irreducible.   Since the $[C]=g([C_0])$ for some $g\in G$, the same holds
for the $g$-conjugate $\widehat{K}$ which acts transitively on $C$. 
If $G_0$ is of Hermitian type 
and 
$\mathcal{M}_D\cong D_{SS}\times \overline{D}_{SS}$,
then this representation is not irreducible. Handling this case requires closer attention.
\subsubsection {Supporting cycle property: The case of an irreducible isotropy representation}
Before going into the details of our construction, let us mention the following
basic fact.  
\begin {lemma}
For every $p\in Z$ it follows that the isotropy group 
$G_p$ acts transitively on $\{ [C]\in \mathcal{M}_Z: p\in C\}=\nu \mu^{-1}(p)$.
\end {lemma}
\begin {proof}
It is enough to prove this for $p\in C_0$.  For this let $C=g(C_0)$ with $p\in C\cap C_0$
and let $\hat{p}\in C_0$ with $g(\hat{p})=p$. Since $C_0$ is a $K$-orbit, there exists
$k\in K$ with $k(p)=\hat{p}$. If $\hat{g}:=gk$, then $\hat{g}(C_0)=C$ and $\hat{g}(p)=p$.
\end {proof}

Now, turning to the construction of supporting cycles,  if $p_0\in \partial Y$, then there exists $[C]\in \partial B$ with $(p_0,[C])\in \partial X$. 
For $p\in C$ we consider the set $\nu\mu^{-1}(p)$ which by the above Lemma can be described as $\{[\hat{C}]:p\in [\hat{C}]\}=G_{p_0}([C])$. 
In particular it is a smooth complex submanifold of $\mathcal{M}_D$ at $[C]$.
We wish to compare the tangent space $T_{[C]}G_p.[C]$ and the complex tangent space $H_{[C]}$ of $\partial B$ at $[C]$.   For this recall that
$H_{[C]}$ is the maximal complex vector subspace of the real tangent space to the boundary $\partial B$ and is a $1$-codimensional complex
subspace of the full tangent space $T_{[C]}\mathcal{M}_D$. 

\bigskip\noindent
The following is a restatement of Theorem \ref{supporting cycle intersection} in the context of the present paragraph.
\begin {proposition} \label{irreducible}
Let $D$ be a flag domain which possesses no non-constant holomorphic functions and let $C$ be a boundary supporting cycle at $p_0\in \partial Y$.  
Then, if $G_0$ is not of Hermitian type or is of Hermitian type with $\mathcal{M}_D\cong D_{SS}$, it follows that $C\cap Y$ contains an open dense
subset of $C$.
\end {proposition} 
\begin {proof}
If for $p\in C$ the stabilizer of $C$ contains the isotropy group $G_p$, then it contains $G_pK$ and is therefore a parabolic subgroup $P$ of $G$.
Consequently, the map $G/G_p\to G/P$ maps $D$ to a flag domain where the base cycle is just a point, i.e., to a HSS.  Since this is not the case
under the assumptions of the proposition, it follows that $G_p[C]$ is non-trivial with tangent space $V_p\not=0$ 
contained in the irreducible $\widehat{K}$-representation space $T_{[C]}\mathcal{M}_D$.  Since the 
complex tangent space $H_{[C]}$ is \emph{not} $\widehat{K}$-invariant, it follows that that there is a dense open subset
of $C$ of points of the form $p=\widehat{k}(p_0)$ with $V_p\not\subset H_{[C]}$, i.e., $V_p+H_{[C]}=T_{[C]}\mathcal{M}_D$.  
For such $p$ it follows that $V_p$ is transversal to $T_{[C]}\partial{B}$ and therefore $G_p.[C]\cap B\not=\emptyset $.  Equivalently, the
$\mu $-fiber at such a point has non-empty intersection with $X$.  In other words, $p\in Y$.
\end {proof}

\subsubsection {Supporting cycle property: The Hermitian case where $\mathcal{M}_D\cong D_{SS}\times \overline{D}_{SS}$}
Our goal here is to prove the analogous result to Proposition 2.3 in the case where $G_0$ is of Hermitian type and 
$\mathcal{M}_D=D_{SS}\times \overline{D}_{SS}$.   Here $\mathcal{M}_Z\cong G/K$, but the isotropy representation of $K$
is not irreducible.   Focusing on the supporting cycle $[C]=g([C_0])$ and the $g$-conjugate $\widehat{K}$ which stabilizes it, we
choose a maximal torus $\widehat{T}$ of $G$ which is contained in $G_{p_0}\cap \widehat{K}$.  Carrying out a root decomposition
with respect to $\widehat{\mathfrak{t}}$, it follows that 
$\mathfrak{g}=\widehat{\mathfrak{u}}_-\oplus \widehat{\mathfrak{k}}\oplus \widehat{\mathfrak{u}}_+$ where 
$\widehat{\mathfrak{p}}_\pm =\widehat{\mathfrak{k}}\oplus \widehat{\mathfrak{u}}_\pm $ are the standard parabolic algebras containing 
$\widehat{\mathfrak{k}}$.  Since $\widehat{T}\subset G_{p_0}$ the tangent space $T_{[C]}G_{p_0}.[C]$ is then realized as a subspace $V_{p_0}$
 which is a direct sum of certain $\widehat{t}$ root algebras in 
$\widehat{\mathfrak{u}}_-\oplus \widehat{\mathfrak{u}}_+$ .  If $V_{p_0}$ were
contained in one or the other of $\widehat{\mathfrak{u}}_\pm $, then $G_{p_0}$ would be contained in one or the other of the standard
parabolic subgroups $P_\pm $.  Denoting this parabolic group by $P$, this would result in a fibration $G/G_{p_0}\to G/P$.  But the
$G_0$-orbit at the point $P$ in the base is an associated Hermitian symmetric space of bounded type.  This is the essential observation
for the proof of the following version of Theorem \ref{supporting cycle intersection} in the context of the present paragraph.
\begin {proposition}
If $D$ is a $G_0$-flag domain with $\mathcal{O}(D)\cong \mathbb{C}$ and $G_0$ is of Hermitian type with 
$\mathcal{M}_D\cong D_{SS}\times \overline{D}_{SS}$,  then every supporting cycle $C$ has the property
that $C\cap Y$ is an open dense subset of $C$.
\end {proposition}
\begin {proof}
The $\widehat{K}$-representations on $\widehat{\mathfrak{u}}_+$ and $\widehat{\mathfrak{u}}_-$ are irreducible.
As was remarked above, since $\mathcal{O}(D)\cong \mathbb{C}$, it follows that both of these nilradicals have non-trivial
intersection with $V_{p_0}$.   Thus $Span\{\widehat{k}(V_{p_0}): \widehat{k}\in \widehat{K}\}$ is the full tangent space
$T_{[C]}\mathcal{M}_D\cong \widehat{\mathfrak{u}}_-\oplus \widehat{\mathfrak{u}}_+$.  Since $\widehat{k}(V_{p_0})=V_{\widehat{k}(p_0)}$,
it follows that for any given complex hyperplane $H$ in $T_{[C]}\mathcal{M}_D$ the set of points $p=\widehat{k}(p_0)$ in $C$
for which $H+V_{p}=T_{[C]}\mathcal{M}_D$ is open and dense in $[C]$.  As in the previous case, this transversality
implies that $p\in Y$.
\end {proof}
\section {A lower bound on the pseudoconcavity}
Let $Y$ be as above and denote by $\partial Y$ the boundary of its interior. If $p_0\in \partial Y$, then
every supporting cycle $C$ at $p_0$ displays a certain degree $c(C,p_0)$ of pseudoconcavity of $Y$ at $p_0$
in the sense that there is a $c(C,p_0)$-dimensional polydisk $\Delta $ contained in $C$ and centered at $p_0$ 
with the property that $\Delta \setminus \{p_0\}$ is contained in $C$.  Furthermore, $c(C,p_0)$ is maximal
with this property.  Then the \emph{supporting cycle pseudoconcavity} of $Y$ at $p_0$ is defined to 
be the maximum the $c(C,p_0)$ as $C$ runs through all supporting cycles containing $p_0$.  Finally,
the pseudoconcavity of $Y$ is defined to be the minimum $c(Y)$ of 
the $c(p_0)$ as $p_0$ runs through $\partial Y$.  As the reader will note, this pseudoconcavity has
very little to do with the exact nature of $Y$, i.e., with the choice of the domain $B$ in $\mathcal{M}_Z$.  In any case the estimate
$c(Y)>d_{ma}$(Theorem \ref{main} and Corollary \ref{pseudoconcave}), which we prove below in the equal rank case, is in terms of the root theoretic invariant
$d_{ma}$ which only depends on the closed $K$-orbit in $Z$. \\

Let us emphasize that we restrict here to the equal rank case, i.e., where the maximal tori of $K$ are also maximal
tori of $G$.  Let $p_0\in \partial Y$ and $[C]$ be a supporting cycle in $c\ell(Y)$ at $p_0$.  The
$G$-isotropy group at $[C]$ is either the appropriate conjugate of $K$ or of a certain
parabolic subgroup $P$ of $G$ which contains $K$.  We denote this group by $J$ and
let $H$ be a complex hyperplane in $\mathfrak{g}/\mathfrak{j}$.  In our discussion here
the conjugation plays no role at all and consequently we may assume that $\mathfrak{j}$ 
is either $\mathfrak{k}$ or $\mathfrak{p}$, depending on the case at hand.   Both contain
a maximal  toral subalgebra $\mathfrak{t}$ which, due to our assumption of maximal rank, is a 
maximal torus in both $\mathfrak{g}$ as well.  The root decompositions which are used
below are with respect to $\mathfrak{t}$. For $\delta $ a $\mathfrak{t}$-root we let $e_\delta $
be such that $\mathbb{C}e_\delta =\mathfrak{g}_\delta $.

\bigskip\noindent
The hyperplane $H$ will be viewed as being in $\mathfrak{g}$ containing $\mathfrak{j}+\mathfrak{q}_0$
where $\mathfrak{q}_0=Lie(Q_0)$ and $Q_0$ is the $G$-isotropy group $G_{p_0}$.  We let 
$\oplus \mathfrak{g}_\alpha=:A$, $\alpha \in \Phi $,  be the $\mathfrak{t}$-invariant complement 
of $\mathfrak{j}+\mathfrak{q}_0$ with $(e_\alpha :\alpha \in \Phi )$ its basis. Denote its dual basis by 
$(f_\alpha:\alpha\in \Phi)$ and let $L=\sum a_\alpha f_\alpha $ be a linear function which has $H$
as its $0$-set.

\bigskip\noindent
Given $p\in C$ we are interested in whether or not the orbit $G_p.[C]$ is tangent to $H$.  For our purposes
it is only necessary to discuss this matter locally near $p_0$; in particular with $p\in U^-.p_0$ where $U^-$ is the opposite
of the unipotent radical of the isotropy group $K_{p_0}$.  For this we identify this orbit with $\mathfrak{u}^-=\oplus \mathfrak{g}_\beta $
for $\beta $ running in the roots $\Lambda (\mathfrak{u}^-)$.  Finally, we let $E:=\oplus \mathfrak{g}_\gamma$, $\gamma \in \Gamma $, be
the $\mathfrak{t}$-root complement of $\mathfrak{j}\cap \mathfrak{q}_0$ in $\mathfrak{q}_0$.
If $exp(\xi )\in U^-$ with $p=Ad(exp(\xi)).p_0$, then $T_{[C]}G_p.[C]\subset H$ if and only if $L(Ad(exp(\xi))(E)))=0$.  Let us refer to
$\{\xi \in A: Ad(exp(\xi))(E)))=0\}$ as the \emph{boundary variety} at $p_0$ and denote it by $S_L$.
Since $Ad(exp(\xi))=e^{ad(\xi)}$, the system of defining equations for $S_L$ which consists of somewhat complicated non-linear equations and is
not easily understandable.Therefore we will be satisfied with determining a linear outer approximation $\widehat{S}_L$ which contains $S_L$. 

\subsection {Linearization}
Let $X$ be a vector field defined by a $1$-parameter group $g_X(t)=exp(\xi t)$ with $\xi=\sum \xi_\beta e_\beta \in \mathfrak{u}^-$. 
The $g_X(t)$-conjugate of the $G_{p_0}$-orbit at $[C]$ is no longer in $H$ if $[X,W]\not\in H$ for some field $W$ corresponding
to an element of $\eta \in E$.  In other words, $Ad(exp(\xi t))(\eta )$ is a curve which satisfies $L(Ad(exp(\xi t)(\eta )))\not=0$ for $t\not=0$
sufficiently small.   Thus if $g_X(t)$ maps the tangent space to the $G_{p_0}$-orbit to a space which is still in $H$, then
$L(exp(\xi t)(\eta ))=0$ for all $\eta \in E$.   Since $L$ is linear, we have an elementary linear approximation.
\begin {proposition} 
In the coordinates defined by $\mathfrak{u}^-$ the boundary variety $S_L$ is contained in
$$
\widehat{S}_L:=\{\xi \in A: L([\xi,\eta])=0 \ \text{for all}\ \eta \in E\}\,.
$$
\end {proposition}
\begin {proof}
$$
\frac{d}{dt}\Big\vert_{t=0}L(Ad(exp(t\xi)(\eta)))=L(\frac{d}{dt}\Big\vert_{t=0}Ad(exp(t\xi)(\eta)))=L([\xi,\eta])\,.
$$
\end {proof}
\subsection {On the maximum dimension of $\widehat{S}_L$}
For a generic linear function $L$ it is difficult to compute $\widehat{S}_L$.  However, for $L=L_\alpha =f_\alpha $ this job can
be translated to a manageable root-combinatoric.  The following fact is therefore quite useful.
\begin {proposition}\label{maximum}
There exits $\alpha \in \Phi$ such that
$$
max_{L}\{dim_\mathbb{C}\widehat{S}_L\}=dim_\mathbb{C}S_{L_\alpha}\,.
$$
\end {proposition}

\noindent
Preparing for the proof,  we first note that in this setting we can view $L$ as varying in projective space so that the maximum is taken on.
Secondly, since $E$ is $T$-invariant, for $t\in T$ it follows that 
$$
t(L)([t(\xi),\eta])=0\ \text{for all} \ \eta \in E \ \Longleftrightarrow \ L([\xi,\eta ])=0 \ \text{for all}\ \eta \in E\,.
$$
Thus $\widehat{S}_{t(L)}=t(\widehat {S}_L)$.\\

\noindent
\textit{Proof of Proposition \ref{maximum}}\\
Let $\eta_1,\ldots \eta_n$ be any basis of $E$ and for $L$ given define $\varphi_L:A\to \mathbb{C}^n$ by 
$\xi\mapsto (L([\xi,\eta_1]),\ldots ,L([\xi, \eta_n]))$.  Then let $\mathcal{L}$ be the space of all $L$ and define
$\mathcal{X}:=\mathcal{L}\times A$ with $\varphi :\mathcal{X}\to \mathbb{C}^n$ defined by 
$\varphi (L,\xi)=\varphi_L(\xi)$.  Finally,  let $\mathcal{K}:=\varphi^{-1}(0)$
and consider the $T$-equivariant surjective projection $\mathcal{K}\to \mathcal{L}$.  Assume that the maximum
dimension $d$ is taken on for the function $L$.  It follows from Proposition \ref{maximum} that $dim_\mathcal{C}S_{t(L)}=d$ for all $t\in T$.
Now let $t(L)$ tend to a point $L_0$ which (in the projective space) is $T$-fixed.  By semicontinuity of fiber
dimension it follows that $dim_\mathcal{C}S_{L_0}\ge d$.  But by the choice of $d$ being maximal, equality
follows.  Consequently, the maximum dimension is taken on by a function $L_0$ with is projectively a $T$-fixed
point  But $L_0=\sum a_\alpha L_\alpha $ transforms by $t(L_0)=\sum a_\alpha \chi_\alpha(t)L_\alpha $ where
$\chi_\alpha $ is the character associated to the root $\alpha $.  Since $\chi_\alpha=\chi_\beta $ if and only if
$\alpha=\beta $, it follows that $L_0=L_\alpha $ for some $\alpha $.\qed
\subsection {Lower bound of pseudoconcavity}
The \emph{attractiveness} of  root $\alpha \in \Phi $ is defined to by
$$
At(\alpha)=\{\beta\in \Lambda (\mathfrak{u}^-): \ \text{there exists}\ \gamma \in \Gamma \ \text{with}\ \gamma +\beta =\alpha\}\,.
$$
Note that if $At(\alpha)=\{\beta_1,\ldots ,\beta_N\}$, then there are unique pairwise different
$\gamma_1,\ldots ,\gamma_N$ with $\alpha=\beta_i+\gamma_i$, $i=1,\ldots ,N$.  Let $R_\alpha$ be the subspace
of $A$ (regarded in the cycle $C$) which is the span of the $e_{\beta_i}$.  By construction it follows that 
$R_\alpha $ has trivial intersection with $\widehat{S}_{L_\alpha}$.  Since the $\beta_i$ are independent, $dim_\mathbb{C}R_\alpha =\vert At(\alpha)\vert$.
Going through the logic, we then have that the maximal dimension $max_L\{dim_\mathbb{C}{S}_L\}$ is achieved by some $S_{L_\alpha}$. 
Its dimension is at most that of $\widehat{S}_{L_{\alpha}}$ which is in turn less than or equal to the
$codim_\mathbb{C}R_\alpha $.  Thus we let $\alpha_{min}$ be any root in $\Phi $ with minimal attractiveness and define
$d_{ma}:=\vert At(\alpha_{min})\vert$.
\begin {theorem} \label{main}
For every $L\in \mathcal{L}$ it follows that $codim_\mathbb{C}S_L\ge d_{ma}\,.$
\end {theorem}
\begin {proof}
Since 
$
R_\alpha \cap S_{L_{\alpha}}=0
$
and $d_{ma}\le dim_\mathbb{C}R_\alpha $ for all $\alpha $, the proof follows immediately from the above results.
\end {proof}

\noindent
Note that $d_{ma}$ is an invariant which only depends on the $K$-orbit $C_0$ in $Z$, i.e., not on the choice of the domain
$B$ in $\mathcal{M}_Z$ at the outset of the construction. 
\begin {corollary}\label{pseudoconcave}
The pseudoconcavity $c(Y) $ is at least the dimension of the span $R_\alpha $ of the root spaces
$\mathfrak{g}_\beta $ for $\beta \in At(\alpha )$ where $\alpha $ is a root of minimum attractiveness, i.e.,
$c(Y)\ge d_{ma}\,.$
\end {corollary}
\noindent
\textbf{Remark.} In closing we should remind the reader that all results in this paragraph were proved under the assumption
that $rank(G)=rank(K)$.

\section{Explicit lower bounds in Grassmannians}
Here we compute explicit lower bounds for the pseudoconcavity of flag domains in Grassmannians in
the equal rank cases of the mixed signature unitary and real symplectic groups.   Although it is possible
that the stabilizer of the base cycle is a parabolic group $P$ containing $K$, since the methods for that
case are exactly the same as those which the stabilizer is $K$, we leave that case to the reader.
\subsection{$SU(p,p^\prime)$-orbits in Grassmannians}
Let $SU(p,p^\prime)$ denote the real form of $SL_{n}(\mathbb{C)}$ consisting
of transformations which preserve the mixed signature Hermitian form given by 
$h(z,w)=z^tE_{p,p^\prime}\bar{w}$ where $E_{p,p^\prime}=Diag(Id_p,-Id_{p^\prime})$.
The flag domains in the Grassmannian 
$Z:=Gr_k(\mathbb{C}^n)$ of $k$-dimensional subspaces of $\mathbb{C}^n$ are
parameterized by pairs $(r,r^\prime)$ of non-negative integers with $r\le p$ and $r^\prime \le p^\prime $:
$D_{r,r^\prime}=\{V\in Z: sign(h\vert V)=(r,r^\prime)\}$.

\bigskip\noindent
As we have seen above, we may choose $K$ in any convenient form and $C_0$ to be its base cycle.
Thus we choose $K$ to have matrices in block form which looks like
\begin {gather*}
\begin {pmatrix}
K_1 & 0\\
0 & K_2
\end {pmatrix}
\end {gather*}
where $K_1$ (resp. $K_2$) is a $p\times p$-matrix (resp. $p^\prime\times p^\prime$-matrix) in a standard basis $(e_1,\ldots ,e_p,f_1,\ldots ,f_{p^\prime})$.
In the above estimate, the dimension of $S_H$ is maximized over all $H$ and consequently we may choose base point $p_0\in C_0$ at our convenience.
We do so by letting $p_0:=Span\{e_1,\ldots, e_r,f_1,\ldots ,f_{r^\prime}\}$.  The unipotent group that we will use 
to move $p_0$ to nearby points looks like
\begin {gather*}
K_1=
\begin {pmatrix}
Id_{r\times r} & 0\\
A  & Id_{(p-r)\times (p-r)}
\end {pmatrix}
\ \text{and} \ K_2=
\begin {pmatrix}
Id_{s\times s} & 0\\
B & Id_{(p^\prime-r^\prime)\times (p^\prime-r^\prime)}
\end {pmatrix}
\end {gather*}
where $A$ and $B$ are arbitrary matrices of the appropriate sizes.  Of course the $K$-isotropy group at the base point is
complementary to these groups.  In addition to the part in $K$ the isotropy group $Q_0$ contains the unipotent group of matrices of the form
\begin {gather*}
\begin {pmatrix}
Id_{p\times p} & U_+\\
U_- & Id_{p^\prime \times p^\prime}
\end {pmatrix}
\end {gather*}
where for example  $U_+$ consists of matrices of the form
\begin {gather*}
\begin {pmatrix}
a & b\\
0 & c
\end {pmatrix}
.
\end {gather*}
Here the entries of the matrices $a $, $b$ and $c$ are arbitrary complex numbers.

\bigskip\noindent
Now our job is to conjugate $U_+$ and $U_-$ with matrices complementary to the isotropy group $K_{p_0}$
and determine what is obtained in the block, e.g., of the $U_+$-picture, where above there is $0$.  For example
for $U_+$ we have
\begin{gather*}
K_1U_+K_2^{-1}=
\begin {pmatrix}
Id & 0\\
A & Id
\end{pmatrix}
\begin{pmatrix}
a & b\\
0  & c
\end{pmatrix}
\begin{pmatrix}
Id & 0\\
-B & Id
\end{pmatrix}
\end{gather*}
so that the relevant block is $A(a -bB)-cB=R$ where $a$,$b$ and $c$ are arbitrary.   For $U_-$ we have an analogous block.

\bigskip\noindent
Now let $\mathfrak{g}_{ij}$ be a spot in the block of relevance of either $U_+$ or $U_-$, i.e., where originally we had $0$. 
The above general result shows that the minimum of the $At(\alpha)$ will be achieved when the linear defining function $f_\alpha $
is just the evaluation at such a spot.  If that $\alpha $ is this root, then the set $S(C)_\alpha $ is defined by the condition
$R_{ij}=0$ for all values of $a$, $b$ and $c$.  For example, this implies that $(Aa)_{ij}=0$ for all $a$ or
equivalently $A_{ik}a_{kj}=0$ for all $a_{kj}$.  In other words the $i$-th row of $A$ vanishes.   Analogously it follows that the $j$-th column of $B$
vanishes.

\bigskip\noindent
The above discussion was for $\mathfrak{g}_{ij}$ and an upper-triangular spot.   For a lower triangular spot the roles of $A$ and $B$ are reversed.
Thus we have the following remark.
\begin {proposition}
If $f_\alpha $ is the root  functional defined by the spot $\mathfrak{g}_{ij}$, then
$codim (S(C)_\alpha)$ is either the length of $A_{i*}$ plus the length of $B_{*j}$ or vice versa.
\end {proposition}
\noindent
In the case where $\mathfrak{g}_{ij}$ is in the upper-triangular piece,  the total codimension (vanishing spots in $A$ and $B$) is then 
$r+p^\prime-r^\prime$.  In the opposite case it is $p-r+p^\prime$.  Thus, applying the estimate derived in the previous paragraph, we
have the following result. 
\begin {corollary}
For the $SU(p,p^\prime)$-flag domain $D_{r,r^\prime}$ it follows that
$$
c(Y)\ge min\{r+p^\prime-r^\prime, p-r+p^\prime\}\,.
$$ 
\end {corollary}

\noindent
\textbf{Remarks.} The above was carried out under the assumption that $\Phi$ parameterizes the complement of $\mathfrak{k}\oplus \mathfrak{q}_0$.
In the case where the stabilizer algebra $\mathfrak{h}$ is a larger parabolic algebra, one obtains a much sharper estimate.  Note that in the cases where
$D$ is a HSS, i.e., $D_{r,0}$ and $D_{0,r^\prime}$, lower concavity bound is, as it should be, zero. This is due to the fact that  $\mathfrak{u}^-=0$.

\subsection{$Sp(2n, \mathbb{R})$-orbits in Grassmannians}
Let $G_0=Aut{(\mathbb{R}^{2n},h)}\cong Sp(2n,\mathbb{R})$ with a non-degenerate alternating form $h$.
Let $Z$ be a Grassmannian of isotropic $k$-planes in $\CC^{2n}$  with $k<n$ and let $D$ be the open $G_0$-orbit of planes of signature $(p,q)$ with $p+q=k$.
We may choose a basis $v_1,\cdots ,v_{2n}$ of $\mathcal{C}^{2n}$ satisfying $v_j=\overline{v_{n+j}}$ and $h(v_j,v_{n+k})=\delta_{jk}$
for $j\leq n$ so that
$$p_{0}=Span_{\mathbb{C}}\{v_1,\cdots, v_p,v_{2n-q+1},\cdots,v_{2n}\}$$
is a base point in $D$.
Let $\mathfrak{q}_0$ be the isotropy subalgebra of $\mathfrak{g}$ at $p_0$.
Then we have
\begin{gather*}
\mathfrak{k}=
\begin{pmatrix}
k&0\\
0&-{}^tk\\
\end{pmatrix}
\ \text{for}\ k\in \Mat(n,n).
\end{gather*}
The nilpotent subalgebra $\mathfrak{u}^-$ is the subalgebra of matrices of the form
\begin {gather*}
k=
\begin {pmatrix}
0 & 0 & 0\\
C_1 & 0 & 0\\
C_2 & C_3 & 0
\end {pmatrix}
\end {gather*}
where
$C_1\in\Mat(n-p-q,p)$ , $C_2\in\Mat{(q,p)}$ and $C_3\in\Mat{(q,n-p-q)}$.

\bigskip\noindent
We then have the decomposition $\frakg=\fraks^+\oplus \frakk\oplus \fraks^-$ where
\begin{align*}
&\fraks^+=\left\{\begin{pmatrix}
0&A^+\\
0&0\\
\end{pmatrix};\; A^+\in \Sym{(n)}\right\},\\
&\fraks^-=\left\{
\begin{pmatrix}
0&0\\
A^-&0\\
\end{pmatrix};\; A^-\in \Sym{(n)}\right\}
\end{align*}
and the parabolic subalgebra $\frakq_0$ decomposes into
$$\frakq_0=(\fraks^+\cap\frakq_0)\oplus (\frakk\cap\frakq_0)\oplus (\fraks^-\cap\frakq_0).$$
Here $\frakq_0\cap \fraks^+$ is the subalgebra of matrices of the form given by
$$
A^+=
\left(
\begin{array}{cc:c}
&&A^+_2\\
\multicolumn{2}{c}{\mbox{\smash{\LARGE $A^+_1$}}}&0\\ \hdashline
{}^tA^+_2&0&0
\end{array}
\right)
\quad \text{where } 
\begin{array}{l}
A^+_1\in\Sym(n-q),\\
A^+_2\in\Mat{(p,q)}.
\end{array}$$
Thus $\frakq_0\cap \fraks^-$ is the subalgebra of matrices of the form given by
$$
A^-=
\left(
\begin{array}{c:cc}
0&0&A^-_2\\
\hdashline
0&&\\ 
{}^tA^-_2&\multicolumn{2}{c}{\mbox{\smash{\LARGE $A^-_1$}}}
\end{array}
\right)
\quad \text{where } 
\begin{array}{l}
A^-_1\in\Sym(n-p),\\
A^-_2\in\Mat{(p,q)}.
\end{array}$$
Let $v^k_{\ell}=v^k\otimes v_{\ell}\in (\CC^{2n})^*\otimes \CC^{2n}$ and let $e_k$ be the dual of $v^k_k-v^{n+k}_{n+k}$ for $1\leq k\leq n$.
Then the set of the roots with respect to the maximal diagonal subalgebra is
$$\Lambda(\frakg)=\{\pm e_j\pm e_k, \pm2e_{\ell}\;|\; 1\leq j< k\leq n,\; 1\leq \ell\leq n \}.$$
Denoting  $\pm (e_j+ e_k)$ (resp. $e_j-e_k$) by $e^\pm_{j,k}$ (resp. $e^c_{j,k}$) and we have
\begin{align*}
&\Lambda(\fraks^{\pm})=\{e^{\pm}_{j,k}\}_{j\leq k}\\
&\Lambda(\fraks^+\cap \frakq_0)=
\left\{\begin{array}{l|l}
e^+_{j,k}&
\begin{array}{l}
1\leq j\leq k\leq n-q
\end{array}
\end{array}\right\}
\cup\left\{\begin{array}{l|l}
e^+_{j,k}&
\begin{array}{l}
1\leq j\leq p,\\
n-q< k\leq n
\end{array}
\end{array}\right\}\\
&\Lambda(\fraks^-\cap \frakq_0)=
\left\{\begin{array}{l|l}
e^-_{j,k}&
\begin{array}{l}
p< j\leq k\leq n
\end{array}
\end{array}\right\}
\cup\left\{\begin{array}{l|l}
e^-_{j,k}&
\begin{array}{l}
1\leq j\leq p,\\
n-q< k\leq n
\end{array}
\end{array}\right\}\\
&\Lambda(\fraku^-)=
\left\{\begin{array}{l|l}
e^c_{j,k}&
\begin{array}{l}
p< j\leq n-q,\\
1\leq k\leq p
\end{array}
\end{array}\right\}
\cup\left\{\begin{array}{l|l}
e^c_{j,k}&
\begin{array}{l}
n-q< j\leq n,\\
1\leq k\leq n-q
\end{array}
\end{array}\right\}
\end{align*}
We set $\Gamma$, $\Phi$, and $At(\alpha)$ for $\alpha\in \Phi$ as the previous section.
\begin{lemma}\label{card}
$\vert At(\alpha)\vert \geq\begin{cases}n-q&\text{if }\alpha\in\Phi\cap\Lambda(\fraks^+),\\ n-p&\text{if }\alpha\in\Phi\cap\Lambda(\fraks^-)\end{cases}$.
\end{lemma}
\begin{proof}
If $\alpha=e_{\ell ,m}^+\in\Phi\cap\Lambda(\fraks^+)$, then $n-q< m\leq n$ and $p<\ell\leq m$.
In this case, $\beta=e_{j,k}^c\in\Lambda(\fraku^-)$ is contained in $At(\alpha)$ if $j= \ell$ or $m$.
A long root $\alpha=e_{\ell,\ell}^+$ minimize $\vert At(\alpha)\vert$ where
$$
At(\alpha)=\{e_{\ell ,1}^c,\cdots ,e_{\ell ,n-q}^c\}.
$$
If $\alpha\in\Phi\cap\Lambda(\fraks^-)$, we analogously have 
$$
At(\alpha)=\{e^c_{p+1,\ell},\cdots ,e^c_{n,\ell}\}
$$
for a long root $\alpha=e_{\ell,\ell}^-$ which minimize $\vert At(\alpha)\vert$.
\end{proof}
\begin{corollary}
$
c(Y)\ge \min{\{n-p,n-q\}}
$ 
\end{corollary}

\end{document}